\documentclass[11pt]{amsart}

\usepackage{a4wide}
\usepackage{amssymb}
\usepackage{mathrsfs}
\usepackage{mathtools}
\usepackage{enumerate}
\usepackage{titletoc}
\usepackage[colorlinks=true, urlcolor=blue, linkcolor=blue, citecolor=blue]{hyperref}
\usepackage[nameinlink]{cleveref}
\usepackage{enumitem}
\usepackage{mathscinet}

\theoremstyle{plain}
\newtheorem{theorem}{Theorem}[section]
\newtheorem{lemma}[theorem]{Lemma}

\newtheorem*{conjecture*}{Conjecture}
\newtheorem*{question*}{Question}

\theoremstyle{definition}
\newtheorem{definition}[theorem]{Definition}
\newtheorem{remark}[theorem]{Remark}

\newtheorem*{notation}{Notation}

\numberwithin{equation}{section}

\makeatletter
\@namedef{subjclassname@2020}{\textup{2020} Mathematics Subject Classification}
\makeatother

\title[Time derivative estimates]{Time derivative estimates for parabolic $p$-Laplace equations and applications to optimal regularity}

\author{Se-Chan Lee}
\address{School of Mathematics, Korea Institute for Advanced Study, Seoul 02455, Republic of Korea}
\email{sechan@kias.re.kr}

\author{Yuanyuan Lian}
\address{Departamento de An\'{a}lisis Matem\'{a}tico,
Instituto de Matem\'{a}ticas IMAG, Universidad de Granada}
\email{lianyuanyuan.hthk@gmail.com; yuanyuanlian@correo.ugr.es}

\author{Hyungsung Yun}
\address{School of Mathematics, Korea Institute for Advanced Study, Seoul 02455, Republic of Korea}
\email{hyungsung@kias.re.kr}

\author{Kai Zhang}
\address{Departamento de Geometr\'{i}a y Topolog\'{i}a,
Instituto de Matem\'{a}ticas IMAG, Universidad de Granada}
\email{zhangkaizfz@gmail.com; zhangkai@ugr.es}



\thanks{Se-Chan Lee has been supported by the KIAS Individual Grant (No. MG099001) at Korea Institute for Advanced Study. Hyungsung Yun has been supported by the KIAS Individual Grant (No. MG097801) at Korea Institute for Advanced Study.
This research has been also financially supported by the Project PID2020-118137GB-I00 funded by MCIN/AEI /10.13039/501100011033.}

\begin{document}

\begin{abstract}
	We establish the boundedness of time derivatives of solutions to parabolic $p$-Laplace equations. Our approach relies on the Bernstein technique combined with a suitable approximation method. As a consequence, we obtain an optimal regularity result with a connection to the well-known $C^{p'}$-conjecture in the elliptic setting. Finally, we extend our method to treat global regularity results for both fully nonlinear and general quasilinear degenerate parabolic problems. 
\end{abstract}
\maketitle

%
%
\section{Introduction}
Partial differential equations that exhibit degeneracies or singularities naturally arise in the modeling of various phenomena, such as the motion of multi-phase fluids in porous media, phase transitions in materials science and option pricing in financial markets. A prototype example, which continues to present significant challenges in regularity theory, is the parabolic $p$-Laplace equations given by
\begin{equation} \label{eq-plaplace}
    u_t =\Delta_p u \coloneqq \mathrm{div}(|Du|^{p-2}Du)
\end{equation}
with $p \in (1, \infty)$. We refer to the monograph by Urbano \cite{Urb08} for further motivation and applications of the parabolic $p$-Laplace equations.

In this paper, we are concerned with the optimal regularity of solutions $u$ to \eqref{eq-plaplace} with $p \in (2, \infty)$. To this end, we establish the boundedness of the time derivative $u_t$ by applying the Bernstein technique together with suitable approximations. To the best of our knowledge, the regularity of $u$ in time has been studied in two different directions: $u \in C_t^{\alpha}$ for some $\alpha \in (0,1)$ or $u_t \in L^2$; see \Cref{history} for a detailed historical background. 
\subsection{Optimal regularity} 
It follows from a direct calculation that
\begin{equation*} 
    \Delta_p (|x|^{p'}) =  n\left(p'\right)^{p-1} \eqqcolon c_{n,p},
\end{equation*}
where $p'\coloneqq p/(p-1)$ is the H\"older conjugate of $p$. This observation gives rise to the following well-known open problem related to the optimal regularity in the elliptic setting, namely $C^{p'}$-regularity conjecture:
\begin{conjecture*}[$C^{p'}$-regularity conjecture]
    Let $p \in (2, \infty)$. Solutions $u$ to 
\begin{equation*}
    -\Delta_p u= f \in L^{\infty}(B_1)
\end{equation*}
are locally of class $C^{1, \frac{1}{p-1}} = C^{p'}$.
\end{conjecture*}
The $C^{p'}$-regularity conjecture was solved completely for $n=2$ and partially for $n \geq 3$ by Ara\'ujo--Teixeira--Urbano in \cite{ATU17} and \cite{ATU18}, respectively, based on the compactness argument. When it comes to the parabolic setting, since $u(x, t)=|x|^{p'} + c_{n,p} t$ is a solution to \eqref{eq-plaplace}, viscosity solutions to \eqref{eq-plaplace} cannot have better regularity than $C^{1, \frac{1}{p-1}}$ in the spatial variables. 

We now claim that the Lipschitz regularity in the time variable of solutions to \eqref{eq-plaplace} is strongly related to the optimal regularity in the spatial variables. To illustrate this, we utilize the scaling argument between the spatial variables and the time variable. Let $u$ be a solution to \eqref{eq-plaplace} and suppose that $Du$ and $u$ have uniform H\"older estimates in the spatial variables and the time variable respectively, that is, 
\begin{equation} \label{scaling_order}
    |Du(x,t)-Du(y,s)|\leq C(|x-y|^\alpha + |t-s|^\mu) \quad \text{and} \quad
    |u(x,t)-u(x,s)|\leq C |t-s|^\nu.
\end{equation}
For given $\alpha \in (0,1)$, let us find $\mu, \nu \in (0,1)$ that are adequate for the scaling relation. 
If we let $v(x,t) \coloneqq (r\rho)^{-1}u(rx,r^2 \rho^{2-p}t)$ for $r>0$ and $\rho >0$, then $v$ is also a solution to \eqref{eq-plaplace} and so $v$ satisfies the uniform estimates \eqref{scaling_order}. By rescaling back to $u$, the uniform estimates \eqref{scaling_order} for $v$ can be written as estimates for $u$: 
\begin{equation*}
   |Du(x',t')-Du(y',s')|\leq C\left(\frac{\rho}{r^\alpha}|x'-y'|^\alpha + \frac{\rho}{(r^2\rho^{2-p})^\mu}|t'-s'|^\mu \right)
\end{equation*}
and
\begin{equation*}
    |u(x', t')-r(x',s')|\leq C \frac{r\rho}{(r^2\rho^{2-p})^\nu}|t'-s'|^\nu,
\end{equation*}
where $x'=rx$, $y'=ry$, $t'=r^2\rho^{2-p}t$ and $s'=r^2\rho^{2-p}s$. Since the uniform estimates \eqref{scaling_order} are invariant with respect to the scaling, we obtain  
\begin{equation*}
    \frac{\rho}{r^\alpha}=1, \quad
    \frac{\rho}{(r^2\rho^{2-p})^\mu}= 1 \quad \text{and} \quad 
     \frac{r\rho}{(r^2\rho^{2-p})^\nu}=1.
\end{equation*}
By taking $\rho=r^\alpha$, we have 
\begin{equation*}
    \mu = \frac{\alpha}{2+\alpha(2-p)} \quad \text{and} \quad 
    \nu = \frac{1+\alpha}{2+\alpha(2-p)}.
\end{equation*}

If the optimal regularity in the spatial variables of solutions to \eqref{eq-plaplace} is proved to be $C^{1,\frac{1}{p-1}}$, that is, $\alpha=1/(p-1)$, then $\nu=1$. Hence, this observation naturally leads to the following question: 
\begin{question*}
    Is the time derivative $u_t$ of a solution $u$ to \eqref{eq-plaplace} bounded?
\end{question*}

Our first main theorem provides a positive answer to this question.
\begin{theorem}\label{thm-plaplace}
    Let $u\in C(Q_1)$ be a viscosity solution to \eqref{eq-plaplace} in $Q_1$ with $p \in (2, \infty)$. Then $u$ is weakly differentiable in time and $u_t \in L^{\infty}_{\mathrm{loc}}(Q_1)$ with the uniform estimate
\begin{equation*}
	\|u_t\|_{L^\infty(Q_{1/2})} \leq C,
\end{equation*}
where $C>0$ is a constant depending only on $n$, $p$ and $\|u\|_{L^{\infty}(Q_1)}$.
\end{theorem}

As a consequence of the uniform $L^{\infty}$-boundedness of $u_t$, we can understand the parabolic $p$-Laplace equations \eqref{eq-plaplace} as elliptic $p$-Laplace equations with nonhomogeneous term $f \coloneqq -u_t \in L^{\infty}(Q_{3/4})$. In other words, \Cref{thm-plaplace} suggests a useful reduction scheme for solutions $u$ to \eqref{eq-plaplace}, which captures a relation between the $C^{p'}$-regularity conjecture in the elliptic and parabolic setting.

To be more precise, we let $\bar\alpha = \bar\alpha (n, p) \in (0,1/(p-1)]$ be the largest number that guarantees solutions $u$ to 
\begin{equation*}
    -\Delta_p u= f \in L^{\infty}(B_1)
\end{equation*}
to be contained in  $C^{1,\bar\alpha}_{\mathrm{loc}}$. For instance, $\bar\alpha(2, p)=1/(p-1)$ by \cite{ATU17}. Then our second main theorem says that the exponent $\bar\alpha$ is also optimal for the parabolic $p$-Laplace equations in the following sense.

\begin{theorem}\label{thm-reduction-quasi}
    Let $u\in C(Q_1)$ be a viscosity solution to \eqref{eq-plaplace} in $Q_1$ with $p \in (2, \infty)$. Then there exists a constant $C>0$ depending only on $n$, $p$ and $\|u\|_{L^{\infty}(Q_1)}$ such that
\begin{equation} \label{op_c1a_PLE}
	|Du(x, t)-Du(y, s)| \leq C\left(|x-y|^{\bar\alpha}+|t-s|^{\frac{\bar\alpha}{1+\bar\alpha}}\right) \quad \text{for all $(x, t), (y,s) \in \overline{Q_{1/2}}$}
\end{equation}
and 
\begin{equation}\label{op_time_PLE}
    |u(x, t)-u(x, t)| \leq C|t-s| \quad \text{for all $(x, t), (x,s) \in \overline{Q_{1/2}}$}.
\end{equation}
\end{theorem}

We recall from Juutinen--Lindqvist--Manfredi \cite[Theorem 4.4]{JLM01} that a continuous function is a viscosity solution to \eqref{eq-plaplace} if and only if it is a weak solution to \eqref{eq-plaplace}. In particular, our main theorems \Cref{thm-plaplace} and \Cref{thm-reduction-quasi} also hold for weak solutions to the parabolic $p$-Laplace equations. Since we focus on the non-variational structure of \eqref{eq-plaplace} without energy-type argument, we adopt the concept of viscosity solutions instead of weak solutions. Moreover, our method works in a rather general non-divergence framework; see \Cref{sec:applications} for applications in different situations.

\subsection{Strategy of the proof}
The proof of time derivative estimates stated in \Cref{thm-plaplace} heavily relies on the Bernstein technique for degenerate operators in non-divergence form. The Bernstein technique plays an important role in establishing boundedness of certain derivatives (e.g., gradient, Hessian, or time derivative) of solutions to elliptic and parabolic equations. In particular, it makes use of the maximum principle for some auxiliary functions to obtain regularity estimates for the solutions. This idea, originated from Bernstein \cite{Ber06, Ber10}, has been widely used in various contexts. For instance, we refer to Caffarelli--Cabr\'e \cite{CC95} and Lee--Lee--Yun \cite{LLY24} for fully nonlinear equations, Wang \cite{Wan98} for mean curvature equations, Athanasopoulos--Caffarelli \cite{AC06} and Fern\'andez-Real--Jhaveri \cite{FRJ21} for thin obstacle problems and Cabr\'e--Dipierro--Valdinoci \cite{CDV22} and Ros-Oton--Torres-Latorre--Weidner \cite{ROTLW25} for nonlocal equations, respectively.

For uniformly parabolic equations with sufficiently smooth coefficients, a key observation of the Bernstein technique is that $u$, $D_k u$ and $u_t$ belong to the same solution class and so $u^2$, $|Du|^2$ and $u_t^2$ are subsolutions to a certain equation; see the books by Lieberman \cite[Lemma 3.18]{Lie96} for linear equations and Dong \cite[Theorem 6 in Chapter IX]{Don91} for fully nonlinear equations. The development of Bernstein technique becomes more challenging in the parabolic $p$-Laplace equations, since (i) the $p$-Laplacian operator $\Delta_p$ lacks a smooth structure and (ii) the first-order term in the equation satisfied by $D_k u$ and $u_t$ exhibit singularities and depend on the higher-order derivatives of $u$. In fact, $D_k u$ and $u_t$ are solutions to the following equation:
\begin{equation*}
    v_t - \Delta_p v = f(Dv, Du, D^2u) \approx |Du|^{p-3}\|D^2 u\| |Dv|.
\end{equation*}

To overcome these difficulties, (i) we consider the regularized $p$-Laplace equation and (ii) instead of using the auxiliary function $v=\eta^2 u_t^2 + \sigma |Du|^2$, which is valid for uniformly parabolic equations, we carefully refine the power of $|Du|$ and the constant $\sigma>0$. The key step of the proof is to compare the powers of singularity in each term and to choose the proper auxiliary function $v$ so that the term with a negative sign dominates all other terms. At this stage, we also convert the dependence of $f$ on $D^2u$ into the dependence on $u_t$ in order to fully exploit the degenerate nature of the equation.

\subsection{Historical background} \label{history}
In 1960s, the celebrated De Giorgi \cite{DeG57}--Nash \cite{Nas58}--Moser \cite{Mos60, Mos61} theory provided a positive answer to Hilbert's 19th problem by establishing the H\"older continuity of solutions to linear uniformly elliptic equations with measurable coefficients. Since the linearity of the operator does not play a crucial role in the theory, the H\"older estimate was extended to weak solutions to the quasilinear elliptic equations, including degenerate and singular elliptic $p$-Laplace equations as model cases, by Lady\v{z}enskaja--Ural'ceva \cite{LU68} and Serrin \cite{Ser64}. We also note that Moser \cite{Mos64} proved that weak solutions to linear uniformly parabolic equations are also locally H\"older continuous.

Nevertheless, the De Giorgi--Nash--Moser theory was not applicable to the degenerate and singular parabolic equations. Instead, DiBenedetto \cite{DiB86} and Chen--DiBenedetto \cite{CD88,CD92} introduced the concept of intrinsic scaling to develop the H\"older continuity of solutions to quasilinear degenerate and singular parabolic equations, respectively. Roughly speaking, the natural space-time scaling in the parabolic $p$-Laplace equations ($p \neq 2$) is determined by the solution itself and so the equation behaves like the heat equation in the view of its own intrinsic geometry. Moreover, the intrinsic scaling approach was employed to deduce the intrinsic Harnack inequality by DiBenedetto \cite{DiB88}, DiBenedetto--Gianazz--Vespri \cite{DGV08, DGV08b, DGV10} and Kuusi \cite{Kuu08} in various situations. We refer to the comprehensive monographs by DiBenedetto \cite{DiB93} and Urbano \cite{Urb08} for more details. 

On the other hand, there were several early works concerning the gradient estimates by Alikakos--Evans \cite{AE83}, the modulus of continuity of gradients by DiBenedetto--Friedman \cite{DF84} and the H\"older gradient estimates by DiBenedetto--Friedman \cite{DF85} and Wiegner \cite{Wie86} for the degenerate and singular parabolic equations. We also refer to the papers by Chen--DiBenedetto \cite{CD89} and Lieberman \cite{Lie93} for the boundary H\"older gradient estimates and to the one by Avelin--Kuusi-Nystr\"om \cite{AKN19} for the boundary Harnack inequality.

Let us next mention relatively recent papers on the regularity of solutions to a class of quasilinear parabolic equations. By utilizing the viscosity method, Jin--Silvestre \cite{JS17} obtained the H\"older gradient estimates of viscosity solutions to the normalized $p$-Laplace equations given by
\begin{equation*}
    u_t=\Delta_p^Nu\coloneqq |Du|^{p-2} \Delta_pu.
\end{equation*}
This approach was then extended to more general class of equations by Imbert--Jin--Silvestre \cite{IJS19} and Fang--Zhang \cite{FZ21, FZ23}. On the other hand, the $W^{2,2}$-regularity of weak solutions $u$ (in particular, $L^2$-regularity of $u_t$) was studied by Dong--Peng--Zhang--Zhou \cite{DPZZ20}; see also H\o{e}g--Lindqvist \cite{HL20} and Feng--Parviainen--Sarsa \cite{FPS23} for similar Sobolev-type regularity results. 

We would like to point out that in all the aforementioned works, solutions to the parabolic $p$-Laplace equations were shown to be $\alpha$-H\"older continuous in time for some $\alpha \in (0,1)$ that cannot be explicitly determined. The novelty of our result is that we establish the Lipschitz continuity of solutions $u$ in the time variable by directly analyzing the time derivative $u_t$. Furthermore, this Lipschitz regularity of $u_t$ implies that the optimal regularity problem in the parabolic setting can be reduced to the well-known $C^{p'}$-conjecture in the elliptic case.

\subsection{Applications}
Our approach used for the time derivative estimate of solutions to the $p$-Laplace equations has several significant applications. To begin with, we study the regularity of viscosity solutions to fully nonlinear degenerate parabolic equations
\begin{equation}\label{eq-fullynonlinear}
    u_t=|Du|^{\gamma}F(D^2u),
\end{equation}
where $\gamma > 0$ and $F$ satisfies \ref{F1}, \ref{F2} that are specified in \Cref{sec:preliminary}. By using the Bernstein technique combined with an another approximation scheme, we again establish the time derivative estimate for $u$; see \Cref{thm-ut-fullynonlinear} for the precise statement. As observed in \Cref{thm-reduction-quasi}, the $L^{\infty}$-boundedness of $u_t$ reveals a strong connection between the spatial regularity of solutions to the parabolic problem and the corresponding elliptic problem. In fact, we exploit the optimal $C^{1, \frac{1}{1+\gamma}}$-regularity developed for the elliptic fully nonlinear framework by Ara\'ujo--Ricarte--Teixeira \cite[Corollary~3.2]{ART15} to derive the following optimal regularity result.
\begin{theorem}\label{thm-fullynonlinear}
    Suppose that $F$ satisfies \textnormal{\ref{F1}} and \textnormal{\ref{F2}}. Let $u\in C(Q_1)$ be a viscosity solution to \eqref{eq-fullynonlinear} in $Q_1$ with  $\gamma >0$. Then there exists a constant $C>0$ depending only on $n$, $p$ and $\|u\|_{L^{\infty}(Q_1)}$ such that
\begin{equation} \label{op_c1a_FNE}
	|Du(x, t)-Du(y, s)| \leq C\left(|x-y|^{\frac{1}{1+\gamma}}+|t-s|^{\frac{1}{2+\gamma}}\right) \quad \text{for all $(x, t), (y,s) \in \overline{Q_{1/2}}$}
\end{equation}
and 
\begin{equation}\label{op_time_FNE}
    |u(x, t)-u(x, s)| \leq C|t-s| \quad \text{for all $(x, t), (x,s) \in \overline{Q_{1/2}}$}.
\end{equation}
\end{theorem}
Let us make a few remarks on \Cref{thm-fullynonlinear}. First of all, if $\gamma=0$, then the celebrated Evans \cite{Eva82}--Krylov \cite{Kry82, Kry83} theory ensures a stronger interior regularity (namely, $C^{2, \alpha}$-regularity for some $\alpha \in (0,1)$) of viscosity solutions. Moreover, by considering a simple viscosity solution
\begin{equation*}
    u(x, t)=C(n,\gamma) t+|x|^{1+\frac{1}{1+\gamma}}, \quad
    C(n,\gamma) \coloneqq \left(1+\frac{1}{1+\gamma}\right)^{1+\gamma}\left(n-1+\frac{1}{1+\gamma}\right)
\end{equation*}
to \eqref{eq-fullynonlinear} with $F(M)=\text{tr}(M)$, we observe that the $C^{1, \frac{1}{1+\gamma}}$-regularity in the spatial variables is optimal indeed. Finally, this regularity result refines and extends the recent result \cite[Theorem~1.1]{LLY24} in several aspects:
\begin{enumerate}[label=(\roman*)]
    \item The $C^{1, \kappa}$-regularity assumption on $F$, which was supposed in \cite[Theorem~1.1]{LLY24}, is no longer required.

    \item \Cref{thm-ut-fullynonlinear} improves the interior regularity of viscosity solutions in both the spatial and time variables, compared to \cite[Theorem~1.1]{LLY24}.

    \item The regularity theory can be developed up to the boundary; see \Cref{subsec-bdry} for details.
\end{enumerate}

We next turn our attention to the boundary regularity of solutions to a class of degenerate parabolic equations. We essentially follow the proof of \Cref{thm-plaplace}, that is, we employ the Bernstein technique by considering some auxiliary functions that involve the solution, its derivatives and cut-off functions. The only minor difference lies in the choice of cut-off function, which does not have a crucial effect on the proof. Therefore, our interior regularity results for both quasilinear and fully nonlinear frameworks can be extended to the global regularity results in a unified manner.

Let us finally point out that our method is also available for solutions $u$ that solve general quasilinear equations 
\begin{equation}\label{eq:generalquasilinear}
		u_t=|Du|^{\gamma}\Delta_p^N u:=|Du|^{\gamma+2-p} \Delta_p u,
\end{equation}
where $\gamma>0$ and $p \in (1, \infty)$. In particular, if we choose $\gamma=p-2$, then \eqref{eq:generalquasilinear} corresponds to the parabolic $p$-Laplace equations \eqref{eq-plaplace}. We note that the condition $\gamma>0$ here is equivalent to the condition $p \in (2, \infty)$ in \Cref{thm-plaplace}. We refer to \Cref{thm-generalquasi} for the precise statement regarding the time derivative estimate in this situation.

\subsection{Overview}
The paper is organized as follows. In \Cref{sec:preliminary}, we introduce the necessary notations, definitions and preliminary results for a class of degenerate parabolic equations. \Cref{sec:plaplace} is devoted to the proof of the main theorems, \Cref{thm-plaplace} and \Cref{thm-reduction-quasi}, concerning the optimal regularity for parabolic $p$-Laplace equations. Finally, in \Cref{sec:applications}, we apply this approach to develop the optimal regularity up to the boundary for both fully nonlinear and quasilinear degenerate parabolic equations.

%
%
\section{Preliminaries}\label{sec:preliminary}
In this section, we summarize some basic notations, definitions and known regularity results that will be used throughout the paper. 
\begin{notation}
Let us display some basic notations as follows.
\begin{enumerate} [label=(\arabic*)]
    \item Sets: 
    \begin{itemize}
        \item $\mathcal{S}^n = \{ M \in  \mathbb{R}^{n\times n} \mid M=M^T\}$.
        \item $B_r(x_0)=\{x\in \mathbb{R}^{n} : |x-x_0|<r\}$ and $B_r=B_r(0)$.
        \item $Q_r(x_0,t_0)=B_r(x_0)\times (t_0-r^{2},t_0]$ and $Q_r=Q_r(0,0)$.
        \item $Q_r^+(x_0,t_0)=Q_r(x_0,t_0)\cap \{x_n>0\}$ and $Q_r^+=Q^+_r(0,0)$.
        \item $Q^\rho_r=B_r\times (-\rho^{2-p}r^2,0]$.
    \end{itemize}
    \item We denotes $\|M\|$ as the Frobenius norm of an $n\times n$ matrix $M$, i.e.,
\begin{equation*}
    \|M\| = \bigg(\sum_{i,j=1}^n |M_{ij}|^2\bigg)^{1/2} \quad \text{for} \quad 
    M=(M_{ij})_{1\leq i,j\leq n}.
\end{equation*}
    \item For constants $0<\lambda\leq \Lambda$, we define the \emph{Pucci's extremal operators} $\mathcal{M}_{\lambda, \Lambda}^{\pm} : \mathcal{S}^n \to \mathbb{R}$ by
    \begin{equation*}
		\mathcal{M}_{\lambda, \Lambda}^{+}(M) \coloneqq \sup_{\lambda I \leq A \leq \Lambda I} \text{tr} (AM) \quad \text{and} \quad
		\mathcal{M}^{-}_{\lambda, \Lambda}(M) \coloneqq \inf_{\lambda I \leq A \leq \Lambda I} \text{tr} (AM).
	\end{equation*}
\end{enumerate}
\end{notation}
We begin with the definition of viscosity solutions to $p$-Laplace equations and fully nonlinear equations, respectively. We refer to \cite{CC95, CIL92} for more details.
\begin{definition}[Viscosity solutions for $p$-Laplace equations]
    Let $p \in (2, \infty)$. A finite almost everywhere and lower semicontinuous function $u :Q_1 \to \mathbb{R} \cup \{+\infty\}$ is a \emph{viscosity supersolution} to \eqref{eq-plaplace} in $Q_1$, if whenever $(x_0, t_0) \in Q_1$ and $\varphi \in C^2(Q_1)$ are such that $u-\varphi$ attains a local minimum at $(x_0, t_0)$, then we have
    \begin{equation*}
        \varphi_t(x_0, t_0)- \Delta_p \varphi(x_0, t_0) \geq 0.
    \end{equation*}
    A finite almost everywhere and upper semicontinuous function $u :Q_1 \to \mathbb{R} \cup \{-\infty\}$ is a \emph{viscosity supersolution} to \eqref{eq-plaplace} in $Q_1$, if whenever $(x_0, t_0) \in Q_1$ and $\varphi \in C^2(Q_1)$ are such that $u-\varphi$ attains a local maximum at $(x_0, t_0)$, then we have
    \begin{equation*}
       \varphi_t(x_0, t_0)- \Delta_p \varphi(x_0, t_0) \leq 0.
    \end{equation*}
    A continuous function $u$ is called a \emph{viscosity solution} to \eqref{eq-plaplace} if and only if it is both viscosity supersolution and subsolution.
\end{definition}

\begin{definition} [Viscosity solutions for fully nonlinear equations]
    Let $\gamma > 0$.
    A finite almost everywhere and lower semicontinuous function $u :Q_1 \to \mathbb{R} \cup \{+\infty\}$ is a \emph{viscosity supersolution} to \eqref{eq-fullynonlinear} in $Q_1$, if whenever $(x_0, t_0) \in Q_1$ and $\varphi \in C^2(Q_1)$ are such that $u-\varphi$ attains a local minimum at $(x_0, t_0)$, then we have
    \begin{equation*}
         \varphi_t(x_0, t_0)-|D\varphi(x_0, t_0)|^{\gamma} F(D^2 \varphi(x_0, t_0)) \geq 0.
    \end{equation*}
    A finite almost everywhere and upper semicontinuous function $u :Q_1 \to \mathbb{R} \cup \{-\infty\}$ is a \emph{viscosity supersolution} to \eqref{eq-fullynonlinear} in $Q_1$, if whenever $(x_0, t_0) \in Q_1$ and $\varphi \in C^2(Q_1)$ are such that $u-\varphi$ attains a local maximum at $(x_0, t_0)$, then we have
    \begin{equation*}
        \varphi_t(x_0, t_0)-|D\varphi(x_0, t_0)|^{\gamma} F(D^2 \varphi(x_0, t_0) ) \leq 0.
    \end{equation*}
    A continuous function $u$ is called a \emph{viscosity solution} to \eqref{eq-fullynonlinear} if and only if it is both viscosity supersolution and subsolution.
\end{definition}

\begin{remark}
    In fact, for singular cases such as $p$-Laplace equations with $p \in (1,2)$ or fully nonlinear equations with $\gamma \in (-1, 0)$, the standard definition of viscosity solutions given in \cite{CC95, CIL92} is not appropriate. In particular, an additional consideration of the vanishing gradient case (i.e., $D\varphi(x_0, t_0)=0$) is required for singular equations; see \cite{FZ23, JLM01, KMP12} for the quasilinear context and \cite{LLY24, OS97} for the fully nonlinear context.
\end{remark}

We next collect some known existence and regularity results that we will need later on. 
\subsection{Regularity results for \texorpdfstring{$p$}{p}-Laplace equations}
To avoid the technical difficulties caused by the non-smoothness of viscosity solutions to \eqref{eq-plaplace} and the $p$-Laplacian operator $\Delta_p$ itself, we introduce the regularization of \eqref{eq-plaplace} as follows:
\begin{equation} \label{eq:regularized} 
   u^{\varepsilon}_t= (\varepsilon^2 +|Du^{\varepsilon}|^2)^{\frac{p-2}{2}}  \left(\delta_{ij} +(p-2) \frac{u^{\varepsilon}_i u^{\varepsilon}_j}{\varepsilon^2+|Du^{\varepsilon}|^2}\right) u^{\varepsilon}_{ij}.
\end{equation}
Then the local boundedness of $Du^{\varepsilon}$ can be obtained by applying Ishii--Lions method \cite[Theorem~III.1]{IL90}.
\begin{lemma}[Uniform Lipschitz estimates, {\cite[Corollary~2.4]{IJS19}}] \label{thm:IJS19}
	Let $u^\varepsilon \in C^\infty(Q_1)$ be a solution to \eqref{eq:regularized} in $Q_1$ with $\varepsilon \in (0,1)$ and $p \in (2, \infty)$. 
    Then 
\begin{equation*}
	|u^\varepsilon(x,t) - u^\varepsilon(y, t)| \leq C |x-y| \quad \text{for all }  (x,t), (y,t) \in Q_{3/4},
\end{equation*}
where $C>0$ is a constant depending only on $n$, $p$ and $\|u^{\varepsilon}\|_{L^\infty(Q_1)}$.
\end{lemma}

The following lemma is required to obtain the family of solutions $\{u^\varepsilon\}_{\varepsilon \in (0,1)}$ to regularized Cauchy–Dirichlet problems is equicontinuous.
\begin{lemma}[Uniform global modulus of continuity, {\cite[Theorem~5.5]{IJS19}}]  \label{lem:mod_conti}
	Let $u^\varepsilon \in C^\infty(Q_1) \cap C(\overline{Q_1})$ be a solution to \eqref{eq:regularized} in $Q_1$
    with $\varepsilon \in (0,1)$ and $p \in (2, \infty)$. 
    Suppose that $u^{\varepsilon}= \varphi$ on $\partial_p Q_1$ and let $\omega$ be a modulus of continuity of $\varphi$. 
    Then there exists a modulus of continuity $\bar{\omega}$ depending only on $n$, $p$, $\omega$ and $\|\varphi\|_{L^\infty(\partial_p Q_1)}$ such that
\begin{equation*}
	|u^\varepsilon(x,t) - u^\varepsilon(y, s)| \leq \bar{\omega} \left(|x-y| + \sqrt{|t-s|}\right) \quad \text{for all }  (x,t), (y,s) \in \overline{Q_1}.
\end{equation*}
\end{lemma}

The next theorem follows from the regularity theory for classical quasilinear equations. 
\begin{theorem}[Solvability of Cauchy--Dirichlet problem, {\cite[Theorem~4.4]{LSU68}}]\label{thm-solvability-quasilinear}
    Let $\varphi\in C(\partial_p Q_1)$. Then, for any $\varepsilon>0$, the regularized Cauchy--Dirichlet problem 
    \begin{equation*}
    \left\{\begin{aligned}
        u^{\varepsilon}_t &= (\varepsilon^2 +|Du^{\varepsilon}|^2)^{\frac{p-2}{2}}  \left(\delta_{ij} +(p-2) \frac{u^{\varepsilon}_i u^{\varepsilon}_j}{\varepsilon^2+|Du^{\varepsilon}|^2}\right) u^{\varepsilon}_{ij} && \text{in } Q_1 \\
        u^{\varepsilon}&=\varphi && \text{on } \partial_p Q_1
    \end{aligned}\right.
    \end{equation*}
    is uniquely solvable in $C^{\infty}(Q_1)\cap C(\overline{Q_1})$.
\end{theorem}
\subsection{Regularity results for fully nonlinear equations}
In \Cref{subsec:FN}, we will discuss the time derivative estimates and the optimal regularity of solutions to fully nonlinear degenerate parabolic equations. For this purpose, we present a fully nonlinear analogue of the results for the $p$-Laplace equations.

We first display two assumptions \ref{F1} and \ref{F2} on the fully nonlinear operator $F:\mathcal{S}^n \to \mathbb{R}$ that will be used throughout the paper.
\begin{enumerate} [label=\text{(F\arabic*)}]
\item \label{F1} $F(0)=0$ and $F$ is $(\lambda,\Lambda)$-\textit{uniformly elliptic}, that is,  there exist constants $0<\lambda \leq \Lambda$ such that for any $M,N \in \mathcal{S}^n$, we have 
	$$\mathcal{M}_{\lambda, \Lambda}^{-}(M-N) \leq F(M) - F(N) \leq \mathcal{M}_{\lambda, \Lambda}^{+}(M-N).$$ 
    
\item \label{F2} $F$ is convex or concave.
\end{enumerate}
On the other hand, the regularization of $F$ is possible through mollification and to do this, we extend the domain of $F$ from $\mathcal{S}^n$ to $\mathbb{R}^{n\times n}$ by considering $F(M)=F\left(\frac{M+M^T}{2}\right)$. Let $\psi \in C_c^{\infty}(\mathbb{R}^{n\times n})$ be a standard mollifier satisfying 
\begin{equation*}
    \int_{\mathbb{R}^{n^2}}\psi \,dM=1 \quad \text{and} \quad \mathrm{supp} \,\psi  \subset \{M \in \mathbb{R}^{n\times n} : \|M \| \leq 1\}.
\end{equation*}
Then we define $F^{\varepsilon}$ as
		\begin{align*}
			F^{\varepsilon}(M) \coloneqq F \ast \psi_{\varepsilon}(M)=\int_{\mathbb{R}^{n\times n}} F(M-N) \psi_{\varepsilon}(N)\,dN,
		\end{align*}
        where $\psi_{\varepsilon}(M)\coloneqq \varepsilon^{-n^2}\psi(M/\varepsilon)$ for $\varepsilon >0$. It is easy to check that $F^{\varepsilon}$ is $(\lambda,\Lambda)$-uniformly elliptic and smooth. Moreover, this mollification preserves the convexity or concavity of $F$. 
    Finally, $F^{\varepsilon}$ converges to $F$ uniformly since $F$ is Lipschitz continuous.

We now regularize the fully nonlinear equations \eqref{eq-fullynonlinear} as follows:
\begin{equation} \label{eq:regularized2} 
   u^{\varepsilon}_t= (\varepsilon^2 +|Du^{\varepsilon}|^2)^{\frac{\gamma}{2}} F^\varepsilon(D^2 u^{\varepsilon}).
\end{equation}
Then uniform local Lipschitz estimates, uniform global modulus of continuity and the solvability of Cauchy--Dirichlet problem are also available in the fully nonlinear setting.
\begin{lemma}[Uniform Lipschitz estimates, {\cite[Lemma~4.2]{LLY24}}] 
	Let $u^\varepsilon$ be a viscosity solution to \eqref{eq:regularized2} in $Q_1$ with $\varepsilon \in (0,1)$ and $\gamma>-2$. 
       Then 
\begin{equation*}
	|u^\varepsilon(x,t) - u^\varepsilon(y, t)| \leq C |x-y| \quad \text{for all }  (x,t), (y,t) \in Q_{3/4},
\end{equation*}
where $C>0$ is a constant depending only on $n$, $\lambda$, $\Lambda$, $\gamma$ and $\|u^{\varepsilon}\|_{L^\infty(Q_1)}$.
\end{lemma}
\begin{lemma}[Uniform global modulus of continuity, {\cite[Lemma~4.11]{LLY24}}]  
	Let $u^\varepsilon \in C^\infty(Q_1) \cap C(\overline{Q_1})$ be a solution to \eqref{eq:regularized2} in $Q_1$
    with $\varepsilon \in (0,1)$ and $\gamma>-1$. Suppose that $u^{\varepsilon}= \varphi$ on $\partial_p Q_1$ and let $\omega$ be a modulus of continuity of $\varphi$. 
    Then there exists a modulus of continuity $\bar{\omega}$ depending only on $n$, $\lambda$, $\Lambda$, $\gamma$, $\omega$ and $\|\varphi\|_{L^\infty(\partial_p Q_1)}$ such that
\begin{equation*}
	|u^\varepsilon(x,t) - u^\varepsilon(y, s)| \leq \bar{\omega} \left(|x-y| + \sqrt{|t-s|}\right) \quad \text{for all }  (x,t), (y,s) \in \overline{Q_1}.
\end{equation*}
\end{lemma}

\begin{theorem}[Solvability of Cauchy--Dirichlet problem, {\cite[Theorem~4.12]{LLY24}}] 
	Let $\varphi\in C(\overline{Q_1})$. Then, for any $\varepsilon>0$, the regularized Cauchy--Dirichlet problem 
    \begin{equation*}  
		\left\{
		\begin{aligned}
			u^{\varepsilon}_t &= (\varepsilon^2+|Du^{\varepsilon}|^2 )^{\frac{\gamma}{2}}F^{\varepsilon}(D^2u^{\varepsilon}) && \text{in $Q_1$}\\
			u^{\varepsilon}&=\varphi && \text{on $\partial_p Q_1$}
		\end{aligned}\right.
    \end{equation*} 
is uniquely solvable in $C^{\infty}(Q_1)\cap C(\overline{Q_1})$.
\end{theorem}

\begin{remark}
If we consider $F^\varepsilon$ instead of $F$ in {\cite[Lemma~3.6 and Lemma~3.7]{LLY24}}, where we employed the Bernstein technique in terms of the difference quotient, then we can drop the condition $F\in C^{1,\kappa} $ in {\cite[Theorem 1.2 and Theorem 4.12]{LLY24}}. Indeed, the estimates in {\cite[Lemma 3.6 and Lemma 3.7]{LLY24}} are independent of the norm $\|F\|_{C^{1,\kappa}}$; that is, we exploit the qualitative (smooth) property of operators, but not the quantitative one. It is noteworthy that the $C^{1, \kappa}$-regularity condition on $F$ is only required for {\cite[Theorem~4.9]{LLY24}} to attain the regularity of small perturbation solutions developed in \cite{Wan13}. 
\end{remark}

For the later use, we also recall the optimal regularity result for viscosity solutions to fully nonlinear degenerate elliptic equations. 
\begin{theorem}[Optimal interior regularity, {\cite[Corollary~3.2]{ART15}}]\label{thm:op_regularity}
    Let $f\in L^\infty(B_1)$ and let $u$ be a viscosity solution to 
    \begin{equation*}
        |Du|^\gamma F(D^2 u) = f \quad \text{in } B_1.
    \end{equation*}
    Then
    \begin{equation*}
        \|u\|_{C^{1,\frac{1}{1+\gamma}}(\overline{B_{1/2}})} \leq C\|u\|_{L^\infty(B_1)},
    \end{equation*}
    where $C>0$ is a constant depending only on $n$, $\lambda$, $\Lambda$, $\gamma$ and $\|f\|_{L^\infty(B_1)}$.  Moreover, this regularity is optimal.
\end{theorem}
We note that the corresponding optimal boundary regularity was recently obtained in \cite{AS23}.

%
%
\section{Parabolic \texorpdfstring{$p$}{p}-Laplace equations}\label{sec:plaplace}
We develop an interior $L^{\infty}$-estimate of the time derivative $u_t$, where $u \in C(Q_1)$ is a viscosity solution to the parabolic $p$-Laplace equations \eqref{eq-plaplace} in $Q_1$ with $p \in (2, \infty)$. The case of $p=2$ corresponds to the standard heat equation (in particular, the equation is no longer degenerate) that immediately yields the smoothness of $u$ in $Q_1$. 

The main strategy is to employ the Bernstein technique for approximating smooth solutions $u^{\varepsilon}$ that solves the regularized $p$-Laplace equation \eqref{eq:regularized}. If we set
\begin{equation*}
    a_\varepsilon^{ij}(q) \coloneqq (\varepsilon^2 +|q|^2)^{\frac{p-2}{2}}  \left(\delta_{ij} +(p-2) \frac{q_i q_j}{\varepsilon^2+|q|^2}\right) \quad \text{for } q \in \mathbb{R}^n,
\end{equation*}
then we observe that
\begin{equation}\label{eq-quasi-ellipticity}
    \lambda_p (\varepsilon^2 +|q|^2)^{\frac{p-2}{2}} |\xi|^2 
    \leq a_\varepsilon^{ij}(q) \xi_i \xi_j 
    \leq \Lambda_p (\varepsilon^2 +|q|^2)^{\frac{p-2}{2}} |\xi|^2 \quad \text{for all } \xi \in \mathbb{R}^n,
\end{equation}
where 
\begin{equation*}
    \lambda_p\coloneqq \min\{p-1,1\} 
    \quad \text{and} \quad
    \Lambda_p \coloneqq \max\{p-1,1\}.
\end{equation*}
\begin{lemma} [Uniform time derivative estimates] \label{ut<C}
	Let $u^\varepsilon \in C(\overline{Q_1}) \cap C^\infty(Q_1)$ be a solution to \eqref{eq:regularized} in $Q_1$ with $\varepsilon \in (0,1)$ and $p \in (2, \infty)$. If $\|Du^\varepsilon\|_{L^{\infty}(Q_1)}\leq 1$, then 
\begin{equation*}
	\|u_t^\varepsilon\|_{L^\infty(Q_{1/2})} \leq C,
\end{equation*}
where $C>0$ is a constant depending only on $n$ and $p$.
\end{lemma}

\begin{proof}
For convenience of notation, we omit $\varepsilon$ from $u^\varepsilon$ and $a_\varepsilon^{ij}$. It immediately follows from \eqref{eq:regularized} that 
\begin{equation} \label{ut_hass}
    |u_t| = |a^{ij}(Du)u_{ij}| \leq \|a^{ij}(Du)\| \|D^2u\| \leq \sqrt{n}\Lambda_p(\varepsilon^2 +|Du|^2)^{\frac{p-2}{2}}\|D^2u\|.
\end{equation}

Moreover, since the partial derivative of $a^{ij}(q)$ with respect to $q_l$, for $l=1,\cdots,n$, is given by
\begin{align*}
    a_{q_l}^{ij}(q)&=\frac{(p-2) q_l a^{ij}(q)}{\varepsilon^2 +|q|^2} + (p-2)(\varepsilon^2 +|q|^2)^{\frac{p-2}{2}} \left(\frac{\delta_{il}q_j+\delta_{jl}q_i}{\varepsilon^2 +|q|^2}- \frac{2q_i q_j q_l}{(\varepsilon^2 +|q|^2)^2}\right),
\end{align*}
it is easy to see that 
\begin{equation} \label{est_avl}
\begin{aligned}
    \|a_{q_l}^{ij}(Du)\| 
    &\leq C_1(\varepsilon^2 +|Du|^2)^{\frac{p-3}{2}}.
\end{aligned}
\end{equation}
for some $C_1=C_1(n, p)>0$. Finally, differentiating \eqref{eq:regularized} with respect to the variables $t$ and $x_k$ gives, respectively: 
\begin{equation*}
    u_{tt} = a^{ij}(Du) u_{ijt} + a_{q_l}^{ij}(Du) u_{ij} u_{lt} \quad \text{and} \quad
    u_{kt} = a^{ij}(Du) u_{ijk} + a_{q_l}^{ij}(Du) u_{ij} u_{lk}. 
\end{equation*}

We next let $A(=A^{\varepsilon}) \coloneqq \|\eta u_t \|_{L^{\infty}(Q_1)}$, where $\eta\in C_c^\infty(Q_1)$ is a nonnegative cut-off function with $\eta\equiv1$ on $\overline{Q_{1/2}}$. Since we are done if $A\leq1$, we may assume that $A>1$ without loss of generality. For constants $\beta \in[0,1)$ and $\delta>0$ to be determined later, we set 
\begin{equation*}
    v\coloneqq \eta^2 u_t^2 + \delta A(\varepsilon^2 + |Du|^2)^{\frac{2-\beta}{2}}    
\end{equation*}
and consider the linearized operator 
\begin{equation*}
    L w\coloneqq a^{ij}(Du) w_{ij} + a_{q_l}^{ij}(Du)u_{ij} w_l
\end{equation*}
of $Q(D^2w, Dw)\coloneqq a^{ij}(Dw)w_{ij}$ at $(D^2u,Du)$. It follows from a direct calculation that
\begin{align*}
    v_t &= 2\eta \eta_t u_t^2 + 2 \eta^2 u_t u_{tt} + (2-\beta)\delta A (\varepsilon^2 + |Du|^2)^{-\frac{\beta}{2}} u_k u_{kt} \\
    v_{i} &=2\eta \eta_i u_t^2 + 2\eta^2 u_t u_{ti} + (2-\beta)\delta A(\varepsilon^2+|Du|^2)^{-\frac{\beta}{2}} u_k u_{ki} \\
    v_{ij} &= 2\eta_i \eta_j u_t^2 + 2\eta \eta_{ij} u_t^2 + 8 \eta \eta_i u_t u_{tj} + 2\eta^2 u_{ti}u_{tj} +2\eta^2 u_t u_{ijt} \\
    &\quad - \beta(2-\beta)\delta A(\varepsilon^2+|Du|^2)^{-\frac{\beta+2}{2}} u_k u_l u_{ki} u_{lj} + (2-\beta)\delta A (\varepsilon^2+|Du|^2)^{-\frac{\beta}{2}} u_{ki} u_{kj} \\
    &\quad+ (2-\beta)\delta A (\varepsilon^2+|Du|^2)^{-\frac{\beta}{2}} u_k u_{ijk}.
\end{align*}
Thus, $v$ satisfies
\begin{equation}\label{vt-Lv}
\begin{aligned} 
    v_t - Lv
    &= 2\eta \eta_t u_t^2\\
    &\quad -2a^{ij}(Du) \left(\eta_i \eta_ju_t^2+\eta^2  u_{ti}u_{tj}+\eta \eta_{ij} u_t^2+4\eta \eta_iu_t  u_{tj}\right)\\
    &\quad -2a^{ij}_{q_l}(Du)\eta \eta_l u_t^2 u_{ij} \\
    &\quad + (2-\beta)\delta A(\varepsilon^2+|Du|^2)^{-\frac{\beta+2}{2}} a^{ij}(Du)\left(\beta u_k u_l  u_{ki} u_{lj} - (\varepsilon^2+|Du|^2)u_{ki} u_{kj} \right). 
\end{aligned}
\end{equation}
Let us estimate each term in \eqref{vt-Lv}. It follows from the ellipticity \eqref{eq-quasi-ellipticity} that
\begin{equation}\label{eq-est-1}
    -2a^{ij}(Du) \left(\eta_i \eta_ju_t^2+\eta^2  u_{ti}u_{tj} \right) \leq -2\lambda_p (\varepsilon^2+|Du|^2)^{\frac{p-2}{2}} \left(|D\eta|^2u_t^2+\eta^2|Du_t|^2 \right)
\end{equation}
and
\begin{equation}\label{eq-est-2}
    -2a^{ij}(Du)\eta \eta_{ij} u_t^2 \leq 2\sqrt{n} \Lambda_p (\varepsilon^2+|Du|^2)^{\frac{p-2}{2}} \|D^2\eta\| \eta u_t^2.
\end{equation}
By applying Young's inequality, we have
\begin{equation}\label{eq-est-3}
\begin{aligned}
   -8a^{ij}(Du) \eta \eta_iu_t  u_{tj}
    &\leq 8\Lambda_p (\varepsilon^2 +|Du|^2)^{\frac{p-2}{2}} \eta|D\eta| |u_t| |Du_t| \\ 
    &\leq (\varepsilon^2 +|Du|^2)^{\frac{p-2}{2}}(2\lambda_p \eta^2|Du_t|^2 + 8\lambda_p^{-1}\Lambda_p^2 |D\eta|^2 u_t^2).
\end{aligned}
\end{equation}
Moreover, the estimates \eqref{ut_hass}, \eqref{est_avl} and the choice of $A$ give  
\begin{equation} \label{eq-est-4}
\begin{aligned}
    -2a^{ij}_{q_l}(Du)\eta \eta_l u_t^2 u_{ij}
    &\leq 2\sqrt{n}\Lambda_pC_1 A(\varepsilon^2 +|Du|^2)^{\frac{2p-5}{2}} |D\eta| \|D^2 u\|^2.
\end{aligned}
\end{equation}
On the other hand, we claim that the following inequality holds:
\begin{equation} \label{eq-est-5}
     a^{ij}(Du) \left(\beta u_k u_l  u_{ki} u_{lj} -(\varepsilon^2+|Du|^2)u_{ki} u_{kj} \right)
     \leq -(1-\beta)\lambda_p (\varepsilon^2 +|Du|^2)^{\frac{p}{2}} \|D^2u\|^2 .
\end{equation}
In fact, by applying the Cauchy--Schwartz inequality, we have
\begin{equation} \label{apply_CSI}
\begin{aligned}
    \sum_{i,j,k,l=1}^n u_k u_l u_{ki} u_{lj}\xi_i\xi_j
    &=\bigg( \sum_{i, k=1}^n u_ku_{ki}\xi_i \bigg)^2=\bigg( \sum_{k=1}^n u_k \Big(\sum_{i=1}^n u_{ki}\xi_i\Big) \bigg)^2\\
    &\leq |Du|^2  \sum_{k=1}^n \Big(\sum_{i=1}^n  u_{ki}\xi_i\Big)^2 = |Du|^2 \sum_{k=1}^n \Big(\sum_{i=1}^n u_{ki}\xi_i\Big) \Big(\sum_{j=1}^n u_{kj} \xi_j\Big) \\
    &=|Du|^2\sum_{i,j,k=1}^n u_{ki} u_{kj}\xi_i\xi_j \leq (\varepsilon^2+|Du|^2)\sum_{i,j,k=1}^n u_{ki} u_{kj}\xi_i\xi_j 
\end{aligned}
\end{equation}
for all $\xi \in \mathbb{R}^n$. Since the matrix $(a^{ij}(Du))_{1 \leq i, j \leq n}$ is positive definite and $\beta \in [0, 1)$, we arrive at \eqref{eq-est-5}.

We then combine \eqref{vt-Lv}-\eqref{eq-est-5} to obtain
\begin{align*}
    v_t - Lv
    &\leq 2\eta |\eta_t| u_t^2+(\varepsilon^2 +|Du|^2)^{\frac{p-2}{2}}\big(-2\lambda_p |D\eta|^2 +2\sqrt{n}\Lambda_p\|D^2\eta\|\eta + 8\lambda_p^{-1}\Lambda_p^2|D\eta|^2 \big)u_t^2 \\
    &\quad +2\sqrt{n}\Lambda_pC_1 A(\varepsilon^2 +|Du|^2)^{\frac{2p-5}{2}} |D\eta| \|D^2 u\|^2\\
    &\quad -\lambda_p(2-\beta)(1-\beta)\delta A (\varepsilon^2+|Du|^2)^{\frac{p-\beta-2}{2}}\|D^2 u\|^2.
\end{align*}
By comparing the exponent of $(\varepsilon^2+|Du|^2)$ in each term and considering the relation $|u_t| \lesssim (\varepsilon^2+|Du|^2)^{\frac{p-2}{2}}\|D^2u\|$, it turns out that the last term dominates the other terms when
\begin{equation*}
        \frac{p-\beta-2}{2} \leq \min \left\{p-2, \frac{3(p-2)}{2}, \frac{2p-5}{2} \right\}=\frac{2p-5}{2}.
\end{equation*}
Therefore, if we choose $\beta=\max\{3-p,0\} \in [0, 1)$, then we can take sufficiently large $\delta=\delta(n, p)>0$ to conclude that
\begin{equation} \label{vt-lv<0}
\begin{aligned}
    v_t - Lv \leq -C_2 \delta A (\varepsilon^2 +|Du|^2)^{-\frac{p+\beta-2}{2}} u_t^2
\end{aligned}
\end{equation}
for some $C_2=C_2(n, p)>0$.

We now suppose that $v$ attains its maximum at some point $(x_0,t_0) \in \overline{Q_1}$ and $|\eta u_t|$ attains its maximum at some point $(x_1,t_1) \in \overline{Q_1}$. If $(x_0, t_0)$ is an interior point, then $v_t(x_0,t_0) -Lv(x_0,t_0) \geq 0$ and hence \eqref{vt-lv<0} implies that $u_t(x_0,t_0)=0$. On the other hand, if $(x_0, t_0)$ is a boundary point, then $\eta(x_0, t_0)=0$ by the choice of the cut-off function.

Thus, in both cases, we have
\begin{equation*}
    A^2 \leq v(x_1,t_1) \leq v(x_0,t_0)= \delta A (\varepsilon^2+|Du(x_0,t_0)|^2)^{\frac{2-\beta}{2}} \leq 2\delta A,
\end{equation*}
which implies the desired uniform estimate
\begin{equation*}
    \|u_t\|_{L^{\infty}(Q_{1/2})} \leq \|\eta u_t\|_{L^{\infty}(Q_1)}=A \leq 2\delta=2\delta(n, p).
\end{equation*}
\end{proof}

We are now ready to prove \Cref{thm-plaplace}.
\begin{proof}[Proof of \Cref{thm-plaplace}]
    Since $u \in C(Q_1) \subset C(\overline{Q_{3/4}})$, \Cref{thm-solvability-quasilinear} guarantees the existence of a unique smooth solution $u^{\varepsilon} \in C(\overline{Q_{3/4}}) \cap C^{\infty}(Q_{3/4})$ to the regularized Cauchy--Dirichlet problem
    \begin{equation*} 
		\left\{
		\begin{aligned}
			u^{\varepsilon}_t&= a^{ij}_{\varepsilon}(Du^{\varepsilon})u^{\varepsilon}_{ij} && \text{in $Q_{3/4}$}\\
			u^{\varepsilon}&=u && \text{on $\partial_p Q_{3/4}$}.
		\end{aligned}\right.
    \end{equation*}
We observe that the maximum principle yields
\begin{equation*}
    \|u^{\varepsilon}\|_{L^{\infty}(Q_{3/4})} \leq \|u\|_{L^{\infty}(Q_{3/4})}.
\end{equation*}
Moreover, it follows from the uniform global modulus of continuity \Cref{lem:mod_conti} and Arzela--Ascoli theorem that there exist a subsequence $\{u^{\varepsilon_k}\}_{k \in \mathbb{N}}$ and a function $\overline{u} \in C(\overline{Q_{3/4}})$ such that $u^{\varepsilon_k} \to \overline{u}$ uniformly in $\overline{Q_{3/4}}$ as $\varepsilon_k \to 0$. By the stability theorem (see \cite[Theorem~2.10]{LLY24} for instance), $\overline{u}$ is a viscosity solution to \eqref{eq-plaplace} and by the comparison principle (see \cite[Theorem~4.7]{JLM01} for instance), we obtain that $u \equiv \overline{u}$ in $\overline{Q_{3/4}}$.

On the other hand, the uniform interior Lipschitz estimate \Cref{thm:IJS19} shows that
\begin{equation*}
    \|Du^{\varepsilon}\|_{L^{\infty}(Q_{1/2})} \leq C_0=C_0\left(n, p, \|u\|_{L^{\infty}(Q_{3/4})}\right).
\end{equation*}
We also observe that if $u^{\varepsilon}$ solves 
    \begin{equation*}
        u^{\varepsilon}_t=(\varepsilon^2+|Du^{\varepsilon}|^2)^{\frac{p-2}{2}} \left( \Delta u^{\varepsilon}+(p-2)\frac{u^{\varepsilon}_iu^{\varepsilon}_j}{\varepsilon^2+|Du^{\varepsilon}|^2} u^{\varepsilon}_{ij}\right) \quad \text{in $Q_{r}^{\rho}=B_r\times (-\rho^{2-p}r^2,0]$}
    \end{equation*}
    for $r=1/2$ and some $\rho>0$ to be determined soon, then $v^{\varepsilon}(x, t) \coloneqq \frac{1}{r\rho}u^{\varepsilon}(rx, r^2\rho^{2-p}t)$ solves
    \begin{equation*}
       v^{\varepsilon}_t=(\varepsilon^2\rho^{-2}+|Dv^{\varepsilon}|^2)^{\frac{p-2}{2}} \left( \Delta v^{\varepsilon}+(p-2)\frac{v^{\varepsilon}_iv^{\varepsilon}_j}{\varepsilon^2\rho^{-2}+|Dv^{\varepsilon}|^2} v^{\varepsilon}_{ij}\right) \quad \text{in $Q_1$}.
    \end{equation*}
    Thus, if we choose $\rho=C_0+1$, then $\|Dv^{\varepsilon}\|_{L^{\infty}(Q_1)} \leq 1$ and so we can apply \Cref{ut<C} to achieve the uniform estimate
    \begin{equation*}
        \|v_t^{\varepsilon}\|_{L^{\infty}(Q_{1/2})} \leq C=C(n, p).
    \end{equation*}
    By scaling back, taking a limit $\varepsilon \to 0$ and using a standard covering argument, we conclude the desired result for $u$.
\end{proof}

By applying the regularity theory for elliptic problems with the reduction argument through \Cref{thm-plaplace}, we obtain regularity results only on hypersurfaces restricted to each time slice. However, by combining H\"older gradient estimates in the spatial variables with H\"older continuity in the time variable, we are able to capture the mixed space-time regularity of the gradient $Du$. It is noteworthy that the following result is independent of the specific choice of equations.

\begin{lemma}\label{lem-simple}
    Let $u \in C(Q_1)$ satisfy 
    \begin{enumerate} [label=(\roman*)]
        \item $u \in C^{1, \alpha}_x$ for some $\alpha \in (0,1]$, i.e.,
    \begin{equation*}
        |u(y, t)- u(x,t) -Du(x, t) \cdot (y-x)| \leq C|x-y|^{1+\alpha}  \quad \text{for all $(x, t), (y, t) \in Q_{1/2}$},
    \end{equation*}
    \item $u \in C^{\beta}_t$ for some $\beta \in (0,1]$, i.e.,
    \begin{equation*}
        |u(x, t)-u(x, s)| \leq C|t-s|^{\beta} \quad \text{for all $(x, t), (x, s) \in Q_{1/2}$}.
    \end{equation*}
    \end{enumerate}
    Then we have
    \begin{equation*}
        |Du(x, t)-Du(x, s)| \leq C|t-s|^{\frac{\alpha \beta}{1+\alpha}} \quad \text{for all $(x, t), (x, s) \in Q_{1/4}$}.
    \end{equation*}
\end{lemma}

\begin{proof}
    Without loss of generality, we may assume $Du(x, t) \neq Du(x, s)$. For given $(x,t), (x,s) \in Q_{1/4}$, we observe that
    \begin{equation*}
        \begin{aligned}
            &|(Du(x, t)-Du(x, s)) \cdot (y-x)|\\
            &\quad \leq |u(y, t)- u(x,t)-Du(x, t) \cdot (y-x)|+|u(y, s)- u(x,s)-Du(x, s) \cdot (y-x)| \\
            &\quad \quad +|u(y, t)-u(y, s)| + |u(x, t)-u(x, s)| \\
            &\quad \leq C(|x-y|^{1+\alpha}+ |t-s|^{\beta})
        \end{aligned}
    \end{equation*}
    for $y \in B_{1/2}$ to be determined soon. If we choose
    \begin{equation*}
        y=x+\frac{|t-s|^{\frac{\beta}{1+\alpha}}}{4} \frac{Du(x, t)-Du(x, s)}{|Du(x, t)-Du(x, s)|} \in B_{1/2},
    \end{equation*}
    then the desired mixed regularity follows.
\end{proof}

We finish this section with the proof of \Cref{thm-reduction-quasi}.

\begin{proof} [Proof of \Cref{thm-reduction-quasi}]
    By \Cref{thm-plaplace}, we have $u_t \in L^{\infty}(Q_{3/4})$ with the estimate \eqref{op_time_PLE}. Then we can interpret $u$ as a viscosity solution to elliptic $p$-Laplace equations
    \begin{equation*}
        -\Delta_p u=f\coloneqq -u_t \in L^{\infty} \quad \text{in $B_{3/4}$.}
    \end{equation*}
    Then the estimate \eqref{op_c1a_PLE} follows from the definition of $\bar\alpha$ and \Cref{lem-simple}.
\end{proof}

%
%
\section{Applications}\label{sec:applications}
We provide several applications of the Bernstein technique and the approximation scheme discussed in the proof of \Cref{ut<C} and \Cref{thm-plaplace}, respectively. Namely, we study the regularity of solutions to fully nonlinear degenerate parabolic equations and general quasilinear degenerate equations, not only in the interior but also up to the boundary.

\subsection{Fully nonlinear equations} \label{subsec:FN}
In this subsection, we investigate the interior regularity of viscosity solutions to fully nonlinear degenerate parabolic equations \eqref{eq-fullynonlinear} with $\gamma>0$. More precisely, our goal is to deduce the following uniform estimates:
\begin{equation*}
	|Du(x, t)-Du(y, s)| \leq C\left(|x-y|^{\frac{1}{1+\gamma}}+|t-s|^{\frac{1}{2+\gamma}}\right) \quad \text{for all $(x, t), (y,s) \in \overline{Q_{1/2}}$}
\end{equation*}
and 
\begin{equation*}
    |u(x, t)-u(x, s)| \leq C|t-s| \quad \text{for all $(x, t), (x,s) \in \overline{Q_{1/2}}$}.
\end{equation*}
Let us begin with the following fully nonlinear counterpart of \Cref{ut<C}.
\begin{lemma} [Uniform time derivative estimates] \label{thm-ut-fullynonlinear}
	Let $u^\varepsilon \in C(\overline{Q_1}) \cap C^\infty(Q_1)$ be a solution to \eqref{eq:regularized2} in $Q_1$ with $\varepsilon \in (0,1)$ and $\gamma >0$. If $\|Du^\varepsilon\|_{L^{\infty}(Q_1)}\leq 1$, then 
\begin{equation*}
	\|u_t^\varepsilon\|_{L^\infty(Q_{1/2})} \leq C,
\end{equation*}
where $C>0$ is a constant depending only on $n$, $\gamma$, $\lambda$ and $\Lambda$.
\end{lemma}

\begin{proof}
For convenience of notation, we omit $\varepsilon$ from $u^\varepsilon$ and $F^{\varepsilon}$. We understand the fully nonlinear degenerate operator $\widetilde{F}(D^2w, Dw)$ as
\begin{equation*}
    \widetilde{F}(M, q)\coloneqq(\varepsilon^2+|q|^2)^{\frac{\gamma}{2}}F(M) \quad \text{for $M=D^2w \in \mathcal{S}^n$ and $q=Dw \in \mathbb{R}^n$}.
\end{equation*}
Then the linearized operator $Lw$ of $\widetilde{F}(D^2w, Dw)$ at $(D^2u,Du)$ is given by
\begin{equation*}
    \begin{aligned}
        L w&\coloneqq  \frac{\partial \widetilde{F}}{\partial M_{ij}}\bigg|_{(D^2u, Du)}  w_{ij} + \frac{\partial \widetilde{F}}{\partial q_l}\bigg|_{(D^2u, Du)} w_l\\
        &=(\varepsilon^2+|Du|^2)^{\frac{\gamma}{2}} F_{ij}(D^2u) \, w_{ij} + \gamma (\varepsilon^2+|Du|^2)^{\frac{\gamma}{2}-1} F(D^2u)\, u_lw_l.
    \end{aligned}
\end{equation*}
By the assumption \ref{F1} (the uniform ellipticity of $F$), we observe that
\begin{equation}\label{eq-ellipticity0}
    \lambda |\xi|^2 \leq F_{ij}(M) \xi_i \xi_j \leq \Lambda |\xi|^2 \quad \text{for all $M \in \mathcal{S}^n$ and $\xi \in \mathbb{R}^n$.}
\end{equation}
From \eqref{eq:regularized2} and \ref{F1}, we obtain 
\begin{equation} \label{eq-ellipticity}
    |u_t| = | (\varepsilon^2+|Du|^2)^{\frac{\gamma}{2}}F(D^2u)| \leq n \Lambda (\varepsilon^2+|Du|^2)^{\frac{\gamma}{2}} \|D^2u\|.
\end{equation}
Moreover, differentiating \eqref{eq:regularized2} with respect to the variables $t$ and $x_k$ gives, respectively: 
\begin{align*}
     u_{tt} &= (\varepsilon^2+|Du|^2)^{\frac{\gamma}{2}}F_{ij}(D^2u)u_{ijt} + \gamma (\varepsilon^2+|Du|^2)^{\frac{\gamma}{2}-1} F(D^2u)\, u_l u_{lt}=Lu_t \quad \text{and} \\
     u_{kt} &= (\varepsilon^2+|Du|^2)^{\frac{\gamma}{2}}F_{ij}(D^2u) u_{ijk} + \gamma (\varepsilon^2+|Du|^2)^{\frac{\gamma}{2}-1} F(D^2u)\, u_lu_{lk}=Lu_k. 
\end{align*}
Let $A \coloneqq \|\eta u_t \|_{L^{\infty}(Q_1)}$ and $\eta\in C_c^\infty(Q_1)$ be a nonnegative cut-off function with $\eta\equiv1$ on $\overline{Q_{1/2}}$. Since we are done if $A\leq1$, we may assume that $A>1$ without loss of generality. For constants $\beta \in [0,1)$ and $\delta>0$ to be determined later, we set 
\begin{equation*}
 v\coloneqq \eta^2 u_t^2 + \delta A (\varepsilon^2+|Du|^2)^{\frac{2-\beta}{2}}.   
\end{equation*}
Then $v$ satisfies
\begin{equation}\label{eq-v}
\begin{aligned} 
    v_t - Lv
    &=2\eta \eta_t u_t^2\\
    &\quad -2(\varepsilon^2+|Du|^2)^{\frac{\gamma}{2}} F_{ij}(D^2u) \left(\eta_i \eta_ju_t^2 +\eta^2u_{ti}u_{tj} +\eta \eta_{ij}u_t^2+ 4 \eta \eta_i u_t u_{tj}\right) \\
    &\quad  -2 \gamma (\varepsilon^2+|Du|^2)^{\frac{\gamma-2}{2}} F(D^2u)\eta \eta_l u_lu_t^2\\
    &\quad + (2-\beta)\delta A (\varepsilon^2+|Du|^2)^{\frac{\gamma-\beta-2}{2}}F_{ij}(D^2u) \left(\beta u_k u_l  u_{ki} u_{lj} -(\varepsilon^2+|Du|^2) u_{ki} u_{kj} \right). 
\end{aligned}
\end{equation}
Let us estimate each term in \eqref{eq-v}. It follows from \eqref{eq-ellipticity0} that
\begin{equation}\label{eq-estimate0}
    -2(\varepsilon^2+|Du|^2)^{\frac{\gamma}{2}} F_{ij}(D^2u)(\eta_i \eta_ju_t^2+\eta^2u_{ti}u_{tj}) \leq -2\lambda(\varepsilon^2+|Du|^2)^{\frac{\gamma}{2}}\left(|D\eta|^2u_t^2+\eta^2|Du_t|^2 \right)
\end{equation}
and
\begin{equation}\label{eq-estimate1}
    -2 (\varepsilon^2+|Du|^2)^{\frac{\gamma}{2}} F_{ij}(D^2u)\eta \eta_{ij}u_t^2 \leq 2\sqrt{n}\Lambda(\varepsilon^2+|Du|^2)^{\frac{\gamma}{2}}\|D^2\eta\|\eta u_t^2.
\end{equation}
By applying Young's inequality, we have
\begin{equation} \label{eq-estimate2}
\begin{aligned}
   -8(\varepsilon^2+|Du|^2)^{\frac{\gamma}{2}} F_{ij}(D^2u)\eta \eta_i u_t u_{tj}  &\leq 8\Lambda (\varepsilon^2 +|Du|^2)^{\frac{\gamma}{2}} \eta|D\eta| |u_t| |Du_t| \\ 
    &\leq (\varepsilon^2 +|Du|^2)^{\frac{\gamma}{2}}(2\lambda \eta^2|Du_t|^2 + 8\lambda^{-1}\Lambda^2 |D\eta|^2 u_t^2).
\end{aligned}
\end{equation}
Moreover, the estimate \eqref{eq-ellipticity} and the choice of $A$ yield  
\begin{equation} \label{eq-estimate3}
\begin{aligned}
    -2 \gamma (\varepsilon^2+|Du|^2)^{\frac{\gamma-2}{2}} F(D^2u)\eta \eta_l u_lu_t^2
    &\leq 2 \gamma A (\varepsilon^2+|Du|^2)^{-\frac{1}{2}} |D\eta| u_t^2.
\end{aligned}
\end{equation}
On the other hand, we apply \eqref{apply_CSI} to have the following inequality.
\begin{equation} \label{eq-claim}
     F_{ij}(D^2u) \left(\beta u_k u_l  u_{ki} u_{lj} -(\varepsilon^2+|Du|^2) u_{ki} u_{kj} \right) \leq -(1-\beta) \lambda (\varepsilon^2+|Du|^2) \|D^2u\|^2. 
\end{equation}
We then combine \eqref{eq-v}-\eqref{eq-claim} to obtain 
\begin{align*}
    v_t - Lv
    &\leq 2\eta|\eta_t|u_t^2\\
    &\quad +(\varepsilon^2+|Du|^2)^{\frac{\gamma}{2}}\left(-2\lambda|D\eta|^2+2\sqrt{n}\Lambda\eta\|D^2\eta\| +8\lambda^{-1}\Lambda^2 |D\eta|^2  \right)u_t^2 \\
    &\quad +2\gamma A (\varepsilon^2+|Du|^2)^{-\frac{1}{2}} |D\eta| u_t^2\\
    &\quad -(2-\beta)(1-\beta) (n\Lambda)^{-2} \delta A(\varepsilon^2+|Du|^2)^{-\frac{\gamma+\beta}{2}}u_t^2.
\end{align*}
By comparing the exponent of $(\varepsilon^2+|Du|^2)$ in each term, it turns out that the last term dominates the other terms when
\begin{equation*}
    -\frac{\gamma+\beta}{2} \leq \min\left\{0,  \frac{\gamma}{2}, -\frac{1}{2}  \right\}=-\frac{1}{2}.
\end{equation*}
Therefore, if we choose $\beta=\max\{1-\gamma,0\} \in [0, 1)$, then we can take sufficiently large $\delta=\delta(n, \gamma, \lambda, \Lambda)>0$ to conclude that
\begin{equation*} 
    v_t-Lv \leq -C \delta A(\varepsilon^2+|Du|^2)^{-\frac{\gamma+\beta}{2}}u_t^2
\end{equation*}
for some $C=C(n, \gamma, \lambda, \Lambda)>0$. The remainder of the proof is the same as in \Cref{ut<C}.
\end{proof}

We next recover the Lipschitz regularity of $u$ with respect to $t$ by taking the limit $\varepsilon \to 0$ in the uniform estimate.
\begin{theorem}\label{cor-fullynonlinear}
	Let $u\in C(Q_1)$ be a viscosity solution to \eqref{eq-fullynonlinear} in $Q_1$ with  $\gamma>0$. Then $u$ is weakly differentiable in time and $u_t \in L^{\infty}_{\mathrm{loc}}(Q_1)$ with the uniform estimate
\begin{equation*}
	\|u_t\|_{L^\infty(Q_{1/2})} \leq C,
\end{equation*}
where $C>0$ is a constant depending only on $n$, $\gamma$, $\lambda$, $\Lambda$ and $\|u\|_{L^{\infty}(Q_1)}$.
\end{theorem}

\begin{proof}
    One can repeat the proof of \Cref{thm-plaplace} with small changes. For example, define an intrinsic cylinder $Q^\rho_r=B_r\times (-\rho^{-\gamma}r^2,0]$, use the corresponding preliminary results in \Cref{sec:preliminary} for fully nonlinear equations instead of quasilinear ones and replace \Cref{ut<C} by \Cref{thm-ut-fullynonlinear}.
\end{proof}

We are now ready to prove the desired interior estimate.
\begin{proof}[Proof of \Cref{thm-fullynonlinear}]
    By \Cref{cor-fullynonlinear}, we have $u_t \in L^{\infty}(Q_{3/4})$ with the estimate \eqref{op_time_FNE}. Then we can understand $u$ as a viscosity solution to fully nonlinear degenerate elliptic equations
    \begin{equation*}
        |Du|^{\gamma}F(D^2u)=f\coloneqq u_t \in L^{\infty} \quad \text{in $B_{3/4}$.}
    \end{equation*}
    For the interior regularity of this type equations, an application of \Cref{thm:op_regularity} yields that $u(\cdot,t) \in C^{1,\frac{1}{1+\gamma}}(\overline{B_{1/2}})$ for all $t \in [-1/4,0]$. Therefore, the estimate \eqref{op_c1a_FNE} follows from \Cref{lem-simple}.
\end{proof}

\subsection{Time derivative estimates on the boundary}\label{subsec-bdry}
This subsection is devoted to the boundary regularity of viscosity solutions to the parabolic $p$-Laplace equations with Dirichlet boundary condition on the flat boundary. To be preicse, for a viscosity solution $u \in C(\overline{Q_1^+})$ to 
\begin{equation} \label{eq-bdry}
	\left\{ \begin{aligned}
		u_t&= \Delta_p u &&\text{in $Q_1^+\coloneqq Q_1 \cap \{x_n>0\}$}\\
		u&=\varphi &&\text{on $\{x_n=0\}$}
	\end{aligned}\right.
\end{equation}
with $p \in (2, \infty)$ and $\varphi \in C^2(\overline{Q_1^+})$, we are interested in the boundary regularity of $u$.

As usual, we first consider the approximating solutions $u^{\varepsilon}$ that solve the following regularized problems:
\begin{equation} \label{eq-bdry-regularized}
	\left\{ \begin{aligned}
		u_t&= a_\varepsilon^{ij}(Du) u_{ij} &&\text{in $Q_1^+$}\\
		u&=\varphi &&\text{on $\{x_n=0\}$},
	\end{aligned}\right.
\end{equation}
where
\begin{equation*}
    a_\varepsilon^{ij}(q) =(\varepsilon^2 +|q|^2)^{\frac{p-2}{2}}  \left(\delta_{ij} +(p-2) \frac{q_i q_j}{\varepsilon^2+|q|^2}\right) \quad \text{for } q \in \mathbb{R}^n
\end{equation*}
as defined in \Cref{sec:plaplace}.
In order to obtain the time derivative estimates for \eqref{eq-bdry-regularized}, we begin with the uniform boundary Lipschitz estimates. 
\begin{lemma} [Uniform boundary Lipschitz estimates] \label{lem-bdry-Lip}
Let $u^\varepsilon \in C^\infty(Q_1^+) \cap C(\overline{Q_1^+})$ be a solution to \eqref{eq-bdry-regularized} with $\varepsilon\in(0,1)$, $p \in (2, \infty)$ and $\varphi \in C^2(\overline{Q_1^+})$. Then 
\begin{equation*}
	|u^{\varepsilon}(x,t)-\varphi(x', 0, t)| \leq Cx_n \quad \text{for all } (x,t) \in \overline{Q_{1/2}^+},
\end{equation*}
where $C>0$ is a constant depending only on $n$, $p$, $\|u^{\varepsilon}\|_{L^\infty(Q_1^+)}$ and $\|\varphi\|_{C^2(\overline{Q_1^+})}$.
\end{lemma}

\begin{proof}
Let
\begin{equation*}
	v(x, t) \coloneqq A(1- |x+e_n|^{-\beta})-At+\varphi(x, t),
\end{equation*}
where $\beta>0$ and $A>0$ will be chosen later. It follows from a direct calculation that
\begin{equation*}
    \begin{aligned}
        v_t&=-A+ \varphi_t,\\
        v_i&=\beta A |y|^{-\beta-2} y_i +\varphi_i,\\
        v_{ij}&=\beta A |y|^{-\beta-2} \delta_{ij}-\beta (\beta+2) A |y|^{-\beta-4} y_i y_j+\varphi_{ij},
    \end{aligned}
\end{equation*}
where $y=x+e_n$. By recalling the ellipticity condition \eqref{eq-ellipticity} and choosing sufficiently large $\beta>0$, we observe that
\begin{equation*}
    \begin{aligned}
        a^{ij}_{\varepsilon}(Dv) v_{ij} 
        &\leq (\varepsilon^2+|Dv|^2)^{\frac{p-2}{2}} \mathcal{M}^+_{\lambda_p,\Lambda_p}(D^2v) \\
        &\leq (\varepsilon^2+|Dv|^2)^{\frac{p-2}{2}} \left(n \Lambda_p \beta A |y|^{-\beta-2} -\lambda_p \beta (\beta+2) A |y|^{-\beta-2} + \mathcal{M}^+_{\lambda_p,\Lambda_p}(D^2 \varphi) \right) \\
        &\leq -C (\beta A)^{p-2} \cdot (\beta^2 A)=-C \beta^pA^{p-1},
    \end{aligned}
\end{equation*}
where $C>0$ depends only on $n$, $p$ and $\|\varphi\|_{C^2(\overline{Q_1^+})}$. We then choose $A>0$ large enough so that $v$ satisfies
\begin{equation*}
	\left\{ \begin{aligned}
		v_t&\geq a_\varepsilon^{ij}(Dv) v_{ij} &&\text{in $Q_1^+$}\\
		v&\geq \varphi &&\text{on $\{x_n=0\}$}\\
        v&\geq \|u^{\varepsilon}\|_{L^{\infty}(Q_1^+)}  &&\text{on $\partial_pQ_1^+ \setminus \{x_n=0\}$}.
	\end{aligned}\right.
\end{equation*}
We now apply the comparison principle between $u^{\varepsilon}$ and $v$ to have
\begin{equation*}
    u^{\varepsilon}(x, t) \leq v(x, t) \leq Cx_n+\varphi(0, 0) \quad \text{on $\{(0', x_n, 0) : 0 \leq x_n \leq 1/2\}$}.
\end{equation*}
By considering a standard translation argument, we obtain the desired upper bound. The lower bound also follows from a similar argument.
\end{proof}

By combining the interior Lipschitz estimates \Cref{thm:IJS19} and the boundary Lipschitz estimates \Cref{lem-bdry-Lip}, we deduce the following global Lipschitz estimates. Since the proof is standard, we omit the proof.
\begin{lemma}
   Let $u^\varepsilon \in C^\infty(Q_1^+) \cap C(\overline{Q_1^+})$ be a solution to \eqref{eq-bdry-regularized} with $\varepsilon\in(0,1)$, $p \in (2, \infty)$ and $\varphi \in C^2(\overline{Q_1^+})$. Then 
\begin{equation*}
	|u^{\varepsilon}(x, t)-u^{\varepsilon}(y, t)| \leq C|x-y| \quad \text{for all $(x, t), (y, t) \in \overline{Q_{1/2}^+}$},
\end{equation*}
where $C>0$ is a constant depending only on $n$, $p$, $\|u^{\varepsilon}\|_{L^\infty(Q_1^+)}$ and $\|\varphi\|_{C^2(\overline{Q_1^+})}$.
\end{lemma}
We are now ready to prove the time derivative estimates up to the boundary based on the Bernstein technique.
\begin{lemma} [Uniform boundary time derivative estimates]
	Let $u^\varepsilon \in C^\infty(Q_1^+) \cap C(\overline{Q_1^+})$ be a solution to \eqref{eq-bdry-regularized} with $\varepsilon\in(0,1)$, $p \in (2, \infty)$ and $\varphi \in C^2(\overline{Q_1^+})$. If $\|Du^\varepsilon\|_{L^{\infty}(Q_1^+)}\leq 1$, then 
\begin{equation*}
	\|u_t^\varepsilon\|_{L^\infty(Q_{1/2}^+)} \leq C,
\end{equation*}
where $C>0$ is a constant depending only on $n$, $p$ and $\|\varphi\|_{C^2(\overline{Q_1^+})}$.
\end{lemma}

\begin{proof}
    The proof is essentially the same as in \Cref{ut<C}, but there are small modifications. To be precise, we take $\eta \in C^\infty(Q_1^+)$ such that $\eta =  1 $ on $\overline{Q_{1/2}^+}$ and $\eta=0$ on $\partial_p Q_1^+ \setminus \{x_n=0\}$. Then the Bernstein technique for the quantity 
    \begin{equation*}
    v=\eta^2u_t^2+\delta A(\varepsilon^2+|Du|^2)^{\frac{2-\beta}{2}}    
    \end{equation*}
    with $A =\|\eta u_t\|_{L^{\infty}(Q_1^+)}$ gives the inequality \eqref{vt-lv<0}, provided that $\delta>0$ is sufficiently large and $\beta=\max\{3-p, 0\}$.
    
    Let $v$ attains its maximum at $(x_0, t_0) \in \overline{Q_1^+}$ and $|\eta u_t|$ attains its maximum at $(x_1, t_1) \in \overline{Q_1^+}$. If $(x_0, t_0)\in Q_1^+$, then we have $u_t(x_0, t_0)=0$ as before. Moreover, if $(x_0, t_0) \in \partial_p Q_1^+ \setminus \{x_n=0\}$, then $\eta(x_0, t_0)=0$ by the choice of $\eta$. Finally, if $(x_0, t_0) \in \{x_n=0\}$, then $u_t=\varphi_t$ from the Dirichlet boundary condition. Therefore, we conclude that
    \begin{equation*}
        A^2 \leq v(x_1, t_1) \leq v(x_0, t_0) \leq \|\varphi_t\|^2_{L^{\infty}(Q_1^+)} +2\delta A.
    \end{equation*}
\end{proof}

\begin{theorem}
    Let $u\in C(Q_1)$ be a viscosity solution to \eqref{eq-bdry} with $p \in (2, \infty)$ and $\varphi \in C^2(\overline{Q_1^+})$. Then $u$ is weakly differentiable in time and we have the uniform estimate
\begin{equation*}
	\|u_t\|_{L^\infty(Q^+_{1/2})} \leq C,
\end{equation*}
where $C>0$ is a constant depending only on $n$, $p$, $\|u\|_{L^{\infty}(Q_1^+)}$ and $\|\varphi\|_{C^2(\overline{Q_1^+})}$.
\end{theorem}

\begin{proof}
    One can follow the proof of \Cref{thm-plaplace}.
\end{proof}

\begin{remark}[Generalization to the fully nonlinear case]
    In fact, it is easy to check that this procedure is also available for the fully nonlinear degenerate framework.    
\end{remark}

\subsection{General quasilinear degenerate equations}
The uniform time derivative estimates for parabolic $p$-Laplace equations with $p\in(2,\infty)$ (\Cref{thm-plaplace}) can be extended to more general quasilinear degenerate equations \eqref{eq:generalquasilinear} with $\gamma>0$ and $p \in (1, \infty)$. 

\begin{theorem}\label{thm-generalquasi}
    Let $u\in C(Q_1)$ be a viscosity solution to \eqref{eq:generalquasilinear} in $Q_1$ with $\gamma>0$ and $p \in (1, \infty)$. Then $u$ is weakly differentiable in time and $u_t \in L^{\infty}_{\mathrm{loc}}(Q_1)$ with the uniform estimate
\begin{equation*}
	\|u_t\|_{L^\infty(Q_{1/2})} \leq C,
\end{equation*}
where $C>0$ is a constant depending only on $n$, $p$, $\gamma$ and $\|u\|_{L^{\infty}(Q_1)}$.
\end{theorem}
It is noteworthy that \Cref{thm-generalquasi} recovers \Cref{thm-plaplace} by considering a special choice of $\gamma=p-2>0$. Since the proof is similar to the one of \Cref{ut<C} or \Cref{thm-ut-fullynonlinear}, we present the sketch of the proof.
\begin{lemma} [Uniform time derivative estimates]  
	Let $u^\varepsilon \in C(\overline{Q_1}) \cap C^\infty(Q_1)$ be a solution to 
    \begin{equation*} 
   u^{\varepsilon}_t= (\varepsilon^2 +|Du^{\varepsilon}|^2)^{\frac{\gamma}{2}}  \left(\delta_{ij} +(p-2) \frac{u^{\varepsilon}_i u^{\varepsilon}_j}{\varepsilon^2+|Du^{\varepsilon}|^2}\right) u^{\varepsilon}_{ij} \quad \text{in } Q_1
\end{equation*}
with $\varepsilon \in (0,1)$, $p \in (1, \infty)$ and $\gamma > 0$. If $\|Du^\varepsilon\|_{L^{\infty}(Q_1)}\leq 1$, then 
\begin{equation*}
	\|u_t^\varepsilon\|_{L^\infty(Q_{1/2})} \leq C,
\end{equation*}
where $C>0$ is a constant depending only on $n$, $p$ and $\gamma$.
\end{lemma}

\begin{proof}
Let us point out several differences compared to the proof of \Cref{ut<C}. Indeed, for
\begin{equation*}
    a_\varepsilon^{ij}(q) \coloneqq (\varepsilon^2 +|q|^2)^{\frac{\gamma}{2}}  \left(\delta_{ij} +(p-2) \frac{q_i q_j}{\varepsilon^2+|q|^2}\right) \quad \text{for } q \in \mathbb{R}^n,
\end{equation*}
a direct calculation implies that
\begin{equation*} 
    |u_t| \leq \sqrt{n}\Lambda_p(\varepsilon^2 +|Du|^2)^{\frac{\gamma}{2}}\|D^2u\|
\end{equation*}
and
\begin{equation*}
    \|a_{q_l}^{ij}(Du)\| \leq C_1(n,p,\gamma)(\varepsilon^2 +|Du|^2)^{\frac{\gamma-1}{2}}.
\end{equation*}
By utilizing these estimates as in the proof of \Cref{ut<C}, we can see that 
\begin{equation*}
    v\coloneqq \eta^2 u_t^2 + \delta A(\varepsilon^2 + |Du|^2)^{\frac{2-\beta}{2}}
\end{equation*}
satisfies 
\begin{align*}
    v_t - Lv
    &\leq 2\eta |\eta_t|u_t^2\\
    &\quad +(\varepsilon^2 +|Du|^2)^{\frac{\gamma}{2}} \left(-2\lambda_p |D\eta|^2 +2\sqrt{n}\Lambda_p \|D^2\eta\| \eta + 8\lambda_p^{-1}\Lambda_p^2|D\eta|^2 \right)u_t^2 \\
    &\quad +2\sqrt{n}\Lambda_p C_1 A (\varepsilon^2 +|Du|^2)^{\frac{2\gamma-1}{2}}   |D\eta| \|D^2u\|^2 \\
    &\quad -\lambda_p (2-\beta)(1-\beta)\delta A (\varepsilon^2 +|Du|^2)^{\frac{\gamma-\beta}{2}}\|D^2 u\|^2.
\end{align*}
By taking $\beta=\max\{1-\gamma,0\} \in [0, 1)$, we have
\begin{equation*}  
    v_t - Lv
    \leq -C_2\delta A (\varepsilon^2 +|Du|^2)^{-\frac{\gamma+\beta}{2}} u_t^2,
\end{equation*}
provided that $\delta=\delta(n,p,\gamma)>0$ sufficiently large. The remainder of the proof is essentially the same as \Cref{ut<C}.
\end{proof}

\noindent{\bf Data availability:} Data sharing is not applicable to this article as no datasets were generated or analyzed during the current study.

%
%

\end{document}